\documentclass[12pt]{amsart}

\usepackage{amscd,amssymb,amsopn,amsmath,amsthm,mathrsfs,graphics,amsfonts,enumerate,verbatim,calc
}

\usepackage[all]{xy}

\usepackage{color}

\usepackage[OT2,OT1]{fontenc}
\newcommand\cyr{%
\renewcommand\rmdefault{wncyr}%
\renewcommand\sfdefault{wncyss}%
\renewcommand\encodingdefault{OT2}%
\normalfont
\selectfont}
\DeclareTextFontCommand{\textcyr}{\cyr}

\usepackage{amssymb,amsmath}

\DeclareFontFamily{OT1}{rsfs}{}
\DeclareFontShape{OT1}{rsfs}{n}{it}{<-> rsfs10}{}
\DeclareMathAlphabet{\mathscr}{OT1}{rsfs}{n}{it}

\numberwithin{equation}{section}
\newtheorem{theorem}{Theorem}[section]
\newtheorem{lemma}[theorem]{Lemma}

\newtheorem{corollary}[theorem]{Corollary}

\theoremstyle{definition}
\newtheorem{definition}[theorem]{Definition}
\newtheorem{remark}[theorem]{Remark}
\theoremstyle{remark}

\newtheorem{example}[theorem]{Example}

\newtheorem{notation}[theorem]{Notation}

\newcommand{\depth}{\operatorname{depth}}

\newcommand{\fm}{\frak{m}}

\newcommand{\fq}{\frak{q}}

\begin{document}
\title[Power of a standard parameter ideal]{On the index of reducibility of powers\\ of a standard parameter ideal}

\author[N.T. Cuong]{Nguyen Tu Cuong}
\address{Institute of Mathematics, 18 Hoang Quoc Viet Road, 10307
Hanoi, Viet Nam}
\email{ntcuong@math.ac.vn}

\author[P.H. Quy]{Pham Hung Quy}
\address{Department of Mathematics, FPT University, Hoa Lac Hi-Tech Park, Ha Noi, Viet Nam}
\email{quyph@fpt.edu.vn}

\author[H.L. Truong]{Hoang Le Truong}
\address{Institute of Mathematics, 18 Hoang Quoc Viet Road, 10307
Hanoi, Viet Nam}
\email{hltruong@math.ac.vn}

\thanks{
This work is partially supported by funds of Vietnam National Foundation for Science
and Technology Development (NAFOSTED) under grant numbers
101.04-2014.15 and 101.04-2014.25. }



\begin{abstract} In this paper we study the index of reducibility of powers of a standard parameter ideal. An explicit formula is proved for the extremely case.  We apply  the main result to compute Hilbert polynomials of socle ideals of standard parameter ideals.
\end{abstract}

\maketitle

\section{Introduction}
Throughout this paper, let $(R, \fm)$ be a Noetherian local ring with the infinite residue field $\frak k = R/\fm$, and $M$ a finitely generalized $R$-module of dimension $d$. A submodule $N$ of $M$ is called an {\it irreducible submodule} if $N$ can not be written as an intersection of two properly larger submodules of $M$. The number of irreducible components of an irredundant irreducible decomposition of $N$, which is independent of the choice of the decomposition  by E. Noether \cite{N21}, is called the {\it index of reducibility} of $N$, and denoted by $\mathrm{ir}_M(N)$. For a parameter ideal $\fq$ of $M$, {\it the index of reducibility} of $\fq$
on $M$ is the index of reducibility of $\fq M$ and denoted by $\mathrm{ir}_M(\fq)$. We have $\mathrm{ir}_M(\fq) = \dim_{\frak k} \mathrm{Soc}(M/\fq M)$. In the case $M$ is Cohen-Macaulay, D.G. Northcott proved that $\mathrm{ir}_M(\fq)$ is an invariant of the module (cf. \cite{No57}) and it is called the {\it Cohen-Macaulay type} of $M$. More precisely, we have $\mathrm{ir}_M(\fq) = \dim_{\frak k}\mathrm{Soc}(H^d_{\fm}(M))$ for all parameter ideals $\fq$, where $H^i_{\fm}(M)$ is the $i$-th local cohomology module of $M$ with respect to the maximal ideal $\frak m$. After that several authors tried to extend Northcott's result for other classes of modules, such as S. Goto, N. Suzuki and H. Sakurai for Buchsbaum modules in \cite{GS03, GS84}; and the authors for generalized Cohen-Macaulay modules in \cite{CQ11, CT08} (see also \cite{GQ16,Q12,Q13,Tr13,Tr16,TY16} for other extensions). If $M$ is a generalized Cohen-Macaulay module and $\fq$ is a {\it standard} parameter ideal of $M$, then Goto and Suzuki in \cite[Theorem 2.1]{GS84} showed that
    $$\mathrm{ir}_M(\frak q) \le \sum_{i=0}^d \binom{d}{i} \dim_{\frak k} \mathrm{Soc}(H^i_{\fm}(M)). \quad (\star)$$
In \cite[Corollary 4.3]{CQ11} and \cite[Theorem 1.1]{CT08} the authors proved that if the  parameter ideal $\frak q$ is contained in a large enough power of $\frak m$ we have
$$\mathrm{ir}_M(\frak q) = \sum_{i=0}^d \binom{d}{i} \dim_{\frak k} \mathrm{Soc}(H^i_{\fm}(M)). \quad (\star \star)$$
On the other hand, for each ideal $I$ the authors in \cite{CQT15} proved that the function $\mathrm{ir}_M(I^{n+1}M)$ becomes a polynomial for large enough $n$. In particular, suppose that $M$ is Cohen-Macaulay and $\fq$ is a parameter ideal, then we have (cf. \cite[Theorem 5.2]{CQT15})
$$\mathrm{ir}_M({\frak q}^{n+1}M)=\binom{n+d-1}{d-1}\dim_{\frak k} \mathrm{Soc}(H^d_{\fm}(M))$$
for all $n \geq 0$. The aim of this paper is to extend this result for the case $M$ is generalized Cohen-Macaulay and $\fq$ is standard. Firstly, similar to the inquality $(\star)$ we have the following result.

\begin{theorem}\label{ThA} Let $M$ be a generalized Cohen-Macaulay module of dimension $d>0$ and $\frak q = (x_1, ..., x_d)$ a standard parameter ideal. Set $s_i(M) = \dim_{\frak k}\mathrm{Soc}(H^i_{\frak m}(M))$ for all $i = 0, ..., d$. Then
 $$\mathrm{ir}_M(\frak q^{n+1}M) \le \sum_{i=1}^{d} \binom{n+d-i}{d-i} \big( \sum_{j=1}^{d-i+1} \binom{d-i}{j-1} s_j(M)\big) + s_0(M)$$
for all $n \ge 0$.
\end{theorem}
Then, when the standard parameter ideal $\fq$ satisfies the extremely condition $(\star \star)$ we have
\begin{theorem}\label{ThB}
  Let $M$ be a generalized Cohen-Macaulay module of dimension
$d>0$ and $\frak q = (x_1, ...,x_d)$ be a standard parameter ideal of $M$. Suppose that $\mathrm{ir}_M(\frak q) = \sum_{i=0}^d \binom{d}{i}s_i(M)$, where $s_i(M) = \dim_{\frak k}\mathrm{Soc}(H^i_{\frak m}(M))$ for all $i = 0, ..., d$. Then
 $$\mathrm{ir}_M(\frak q^{n+1}M) = \sum_{i=1}^{d} \binom{n+d-i}{d-i} \big( \sum_{j=1}^{d-i+1} \binom{d-i}{j-1} s_j(M)\big) + s_0(M)$$
for all $n \ge 0$.
\end{theorem}
It should be noted here that the technical arguments that used in this paper are mainly based on technical properties of standard system of parameters and colon ideals. For example, the key ingredient of the proof of Theorem \ref{ThB} is proving that the equality $\frak q^{n+1}M :_R \frak  m  = \frak q^{n} (\frak qM :_R \frak m)
+ (0 :_M \frak  m)$ holds true for all $n \geq 0$ (Lemma \ref{L4.5}).\\

This paper is organized as follows. In the next section we recall the notions of generalized Cohen-Macaulay, standard parameter ideal and the index of reducibility. Theorem \ref{ThA} is proved in Section 3. Section 4 is devolved to prove Theorem \ref{ThB}. In the last section, we apply Theorem \ref{ThB} to give an explicit description for the Hilbert polynomials of socle ideals (cf. Theorem \ref{T4.10}).

\section{Preliminary}
We start this section with the notions of generalized Cohen-Macaulay and standard parameter ideals in terms of local cohomology (cf. \cite{CST78,T86}).
\begin{definition}\rm
\begin{enumerate}[{(i)}]
\item An $R$-module $M$ is called a {\it generalized Cohen-Macaulay} module if $H^i_{\frak m}(M)$ has finite length for all $i<d$.
\item A parameter ideal $\frak q = (x_1, ..., x_d)$ is called {\it standard} if
$$\frak q H^i_{\frak m}(M/(x_1,...,x_j)M) = 0$$
for all non-negative integers $i, j$ with $i +j <d$.
\item An $R$-module $M$ is called {\it Buchsbaum} if every parameter ideal is standard.
\end{enumerate}
\end{definition}
Notice that a standard system of parameters forms a $d$-sequence. The following result is useful in this paper (see \cite[Corollary 2.6]{T86}).
\begin{lemma} \label{L2.2} Let $\frak q = (x_1, ..., x_d)$ be a standard parameter ideal of $M$. Set $\frak q_i = (x_1, ..., x_i)$ for all $i = 0, ..., d-1$. Then we have
  \begin{enumerate}[{(i)}] \rm
  \item {\it $(\frak q^{n+1}, \frak q_{i-1})M : x_i = \frak q^nM + (\frak q_{i-1}M : x_i)$ {\it for all } $n>0$, $i = 1, .., d$.}
  \item {\it $(\frak q_{i-1}M : x_i) \cap \frak qM =  \frak q_{i-1}M$ {\it for all } $i = 1, .., d$.}
  \end{enumerate}
\end{lemma}
We now present the main object of this paper.
\begin{definition}\rm A submodule $N$ of $M$ is called an {\it irreducible submodule} if $N$ can not be written as an intersection of two properly larger submodules of $M$. The number of irreducible components of an irredundant irreducible decomposition of $N$, which is independent of the choice of the decomposition, is called the {\it index of reducibility} of $N$, and denoted by $\mathrm{ir}_M(N)$. For a parameter ideal  $\frak q$ of $M$, we define the index of reducibility of $\frak q$ on $M$ is the index of reducibility of $\frak qM$, and denoted it by $\mathrm{ir}_M(\frak q)$.
\end{definition}
\begin{remark}\rm
We denoted by $\mathrm{Soc}(M)$ the sum of all simple submodules of $M$. $\mathrm{Soc}(M)$ is called the socle of $M$. If $R$ is a local ring with the unique maximal ideal $\frak m$ and $k = R/\frak m$ its residue field, then it is well-known that $\mathrm{Soc}(M) = 0:_M\frak m$ is a $k$-vector space of finite dimension. Let $N$ be a submodule of $M$ with $\ell_R(M/N) < \infty$. Then it is easy to check that $\mathrm{ir}_M(N) = \ell_R((N:\frak m)/N) = \dim_{\frak k} \mathrm{Soc}(M/N).$
\end{remark}
\begin{notation}\rm In this paper, for each $i = 0, ..., d$, we set $s_i(M) = \dim_{\frak k} \mathrm{Soc}(H^i_{\frak m}(M))$.
\end{notation}

\begin{remark}\rm \label{R2.7}
\begin{enumerate}[{(i)}]
\item If $M$ is Cohen-Macaulay, a well-known result of Northcott said that the index of reducibility of $\frak q$ on $M$ is an invariant of the module, More precisely, $\mathrm{ir}_M(\frak q) = s_d(M)$ for all parameter ideals $\frak q$.
\item This Northcott's result was considered by many authors in larger classes of modules (see, \cite{CQ11, CT08, GS03,GS84,Q12,Q13,Tr13}). Recently, it is extended for any finite generated $R$-module in \cite{GQ16}. In the case $M$ is generalized Cohen-Macaulay, Goto and Suzuki in \cite[Theorem 2.1]{GS84} showed that
    $$\mathrm{ir}_M(\frak q) \le \sum_{i=0}^d s_i(M) \binom{d}{i}$$
    for every standard parameter ideal $\frak q$ of $M$. In \cite[Corollary 4.3]{CQ11} and \cite[Theorem 1.1]{CT08} the authors proved that the equality occurs when $\frak q$ contained in a lager enough power of $\frak m$.
\end{enumerate}
\end{remark}
For each ideal $I$, it is natural to ask about the behavior of the function $\mathrm{ir}_M(I^{n+1}M)$ in terms of $n$. In \cite[Theorems 4.1]{CQT15} the authors proved the following theorem (see also \cite{Tr14, Tr15}).

\begin{theorem}\label{T2.8} Let $(R, \fm)$ be a Notherian local ring and $M$ a finitely generated $R$-module of dimension $d$. For each ideal $I$ the function $\mathrm{ir}_M(I^{n+1}M)$ becomes a polynomial for large enough $n$. Furthermore, if $I$ is an ideal such that $\ell(M/IM) < \infty$, then the polynomial has degree $d-1$ and written as follows
  $$\mathrm{ir}_M(I^{n+1}M)=\sum\limits_{i=0}^{d-1} (-1)^{i}f_i(I;M)\binom{n+d-i-1}{d-i-1}$$
for large enough $n$ and $f_i(I;M) \in \mathbb{Z}$ for all $i = 0, \ldots, d-1$.
\end{theorem}
In the case of parameter ideals, we have the following (cf. \cite[Theorem 5.2]{CQT15})
\begin{theorem}
Let $M$ be a Cohen-Macaulay $R$-module of dimension $d$ and $\fq$ a parameter ideal of $M$. Then we have
    $$\mathrm{ir}_M({\frak q}^{n+1}M)=s_d(M) \binom{n+d-1}{d-1}$$
    for all $n \geq 0$.
\end{theorem}
The readers may find in \cite{Tr14,Tr15, TY16} for more characterizations of the Cohen-Macaulayness of $M$ in terms of the coefficient $f_0(\fq, M)$. In this paper we study the function $\mathrm{ir}_M(\fq^{n+1}M)$ when $M$ is generalized Cohen-Macaulay and $\fq$ is a standard parameter ideal of $M$.

\section{An upper bound formula}
In this section we estimate the index of reducibility of powers of standard parameter ideals.  By the next result we show that the problem can be reduced to the case $\mathrm{depth}(M) > 0$.
\begin{lemma}\label{L3.1}
Let $M$ be a generalized Cohen-Macaulay module of dimension
$d$ and $\overline{M} = M/H^0_{\frak m}(M)$. Let $\frak q = (x_1, ...,x_d)$ be a standard parameter ideal of $M$. Then
$$\mathrm{ir}_M(\frak q^{n+1}M) = \mathrm{ir}_{\overline{M}}(\frak q^{n+1}\overline{M}) + s_0(R)$$
for all $n \ge 1$.
\end{lemma}
\begin{proof}
Since $\frak q$ is standard and Lemma \ref{L2.2} we have
$$(\frak q^{n+1} + H^0_{\frak m}(M)):_M \frak  m \subseteq (\frak q^{n+1}M + H^0_{\frak m}(M)):_M x_1 = \frak q^{n}M + H^0_{\frak m}(M).$$
Let $x \in (\frak q^{n+1}M + H^0_{\frak
m}(M)):_R \frak  m$, we have $x = y + z$ with $y \in \frak q^{n}M$
and $z \in H^0_{\frak m}(M)$. Hence $\frak m y \subseteq \frak q^{n}M
\cap (\frak q^{n+1}M + H^0_{\frak m}(M)) = \frak q^{n+1}M$. So $y \in \frak
q^{n+1}M :_M \frak  m$. Therefore
$$(\frak q^{n+1}M + H^0_{\frak m}(M)):_M \frak  m = (\frak q^{n+1}M :_M \frak m )+ H^0_{\frak m}(M)$$
for all $n\ge 1$. From the above equation we have that the short exact
sequence
$$0 \to H^0_{\frak m}(M) \to M/\frak q^{n+1}M \to \overline{M}/\frak q^{n+1}\overline{M} \to
0$$
derives the short exact sequence
$$0 \to (0):_{H^0_{\frak m}(M)}\frak m \to (\frak q^{n+1}M:_M\frak m)/\frak q^{n+1}M \to (\frak q^{n+1}\overline{M}:_{\overline{M}}\frak m)/\frak q^{n+1}\overline{M} \to
0$$ for all $n \ge 1$. The proof is complete.
\end{proof}

When $\depth(M) > 0$ we have $s_0(M) = 0$. So the following lemma is obvious, hence we omit its proof.
\begin{lemma}\label{L3.2} Let $M$ be a generalized Cohen-Macaulay module of dimension
$d>0$ and $\depth(M) > 0$. Then every parameter element $x$ is a regular element of $M$ and $s_1(M) = \ell((xM:\fm)/xM)$.
\end{lemma}

For inductive arguments that used in this paper we need the following result.
\begin{lemma}\label{L3.3}
  Let $M$ be a generalized Cohen-Macaulay module of dimension
$d>1$ and $\mathrm{depth}(M)>0$. Let $\frak q = (x_1, ...,x_d)$ be a standard parameter ideal of $M$ and set $M' = M/x_1M$ and $\frak q' = (x_2, ..., x_d)$. Then
$$\mathrm{ir}_M(\frak q^{n+1}M) - \mathrm{ir}_M(\frak q^{n}M) \le \mathrm{ir}_{M'}((\frak q')^{n+1}M') - s_1(M)$$
for all $n \ge 1$.
\end{lemma}
\begin{proof} By Lemma \ref{L2.2} we have $\frak q^{n+1}M : x_1 = \frak q^nM$, so we have the short exact sequence
$$0 \to M/\frak q^nM \to M/\frak q^{n+1}M \to M'/(\frak q')^{n+1}M' \to 0$$
  for all $n \ge 1$. By applying the functor
$\mathrm{Hom}_R(R/\frak m, \bullet)$ we obtain the following exact sequence
$$0 \to (\frak q^nM:_{M}\frak m)/\frak q^nM \to
(\frak q^{n+1}M:_M\frak m)/\frak q^{n+1}M \overset{\varphi_n}{\to}
[(\frak q')^{n+1}M' :_{M'} \frak  m]/(\frak q')^{n+1}M' .$$
Therefore
$$\mathrm{ir}_M(\frak q^{n+1}M) - \mathrm{ir}_M(\frak q^{n}M) = \mathrm{ir}_{M'}((\frak q')^{n+1}M') - \ell (\mathrm{coker}(\varphi_n))$$
for all $n \ge 1$. On the other hand we have $\mathrm{im}(\varphi_n) = [(\frak q^{n+1}M:_M\frak m)+ x_1M]/(\frak q^{n+1}M+x_1M)$, hence
$$\mathrm{coker}(\varphi_n) = \frac{(\frak q^{n+1}M+x_1M):_M\frak m}{(\frak q^{n+1}M:_M \frak m)+ x_1M}.$$
Thus for all $n \ge 1$ we have
\begin{eqnarray*}
\ell(\mathrm{coker}(\varphi_n)) &\ge & \ell \big( \frac{(\frak q^{n+1}M:_M \frak m)+ (x_1M:_M \frak m)}{(\frak q^{n+1}M:_M \frak m)+ x_1M} \big)\\
&=& \ell \big( \frac{ x_1M:_M \frak m}{( x_1M:_M \frak m) \cap [(\frak q^{n+1}M:_M \frak m)+ x_1M]} \big)\\
&=& \ell \big( \frac{ x_1M:_M \frak m}{[( x_1M:_M \frak m) \cap (\frak q^{n+1}M:_M \frak m)]+ x_1M} \big)\\
&=& \ell \big( \frac{ x_1M:_M \frak m}{x_1M} \big)\\
&=& s_1(M).
\end{eqnarray*}
For the third equation we note that
$$( x_1M:_M \frak m) \cap (\frak q^{n+1}M:_M \frak m) \subseteq ( x_1M:_M x_2) \cap \frak q^nM \subseteq x_1M$$ by Lemma \ref{L2.2}, and the last equation follows from Lemma \ref{L3.2}. The proof is complete.
\end{proof}
We are now ready to prove the main result of this section.
\begin{proof}[Proof of Theorem \ref{ThA}]
We prove by induction of $d$. If $d=1$ we have $\frak q = (x)$ and $\frak q^{n+1} = (x^{n+1})$. The assertion follows from the Goto-Suzuki result (see, Remark \ref{R2.7} (ii)). Suppose that $d>1$ and the assertion is proved for $d-1$. If $n=0$, thanks to Goto-Suzuki's result we have
  $$\mathrm{ir}_M(\frak q) \le \sum_{i = 0}^d \binom{d}{i}s_i(M) = \sum_{i=1}^{d}\big( \sum_{j=1}^{d-i+1} \binom{d-i}{j-1} s_j(M)\big) + s_0(M).$$
 For $n \ge 1$, by Lemma \ref{L3.1} we have $\mathrm{ir}_M(\frak q^{n+1}M) = \mathrm{ir}_{\overline{M}}(\frak q^{n+1}\overline{M}) + s_0(M)$ where $\overline{M} = M/H^0_{\frak m}(M) $. Since $s_0(\overline{M}) = 0$ and $s_i(\overline{M}) = s_i(M)$ for all $i>0$, we can assume henceforth that $\mathrm{depth}(M) > 0$. We need to prove that
 $$\mathrm{ir}_M(\frak q^{n+1}M) \le \sum_{i=1}^{d} \binom{n+d-i}{d-i} \big ( \sum_{j=1}^{d-i+1} \binom{d-i}{j-1} s_j(M)\big ).$$
 Set $M' = M/(x_1)M$, we have $\frak q' = (x_2, ..., x_d)$ is a standard parameter ideal of $M'$. The short exact sequence
 $$0 \to M \overset{x_1}{\to} M \to M' \to 0$$
 derives that $H^0_{\frak m}(M') \cong H^1_{\frak m}(M)$ and the exact sequence
 $$0 \to H^i_{\frak m}(M) \to H^i_{\frak m}(M') \to H^{i+1}_{\frak m}(M) $$
 for all $i = 1, ..., d-1 $. Therefore $s_0(M') = s_1(M)$ and $s_i(M') \le s_i(M) + s_{i+1}(M)$ for all $i = 1, ..., d-1$. By Lemma \ref{L3.3} we have
 $$\mathrm{ir}_M(\frak q^{n+1}M) - \mathrm{ir}_M(\frak q M)  \le  \sum_{k=1}^n [\mathrm{ir}_{M'}((\frak q')^{k+1}M) - s_1(M)].$$
By inductive hypothesis we have
\begin{eqnarray*}
\mathrm{ir}_M(\frak q^{n+1}M) &-& \mathrm{ir}_M(\frak q) \\
&\le&  \sum_{k=1}^n \bigg[\sum_{i=1}^{d-1} \binom{k+d-i-1}{d-i-1} \bigg(
\sum_{j=1}^{d-i} \binom{d-i-1}{j-1} s_j(M')\bigg) + s_0(M') -
s_1(M) \bigg]\\
&\le & \sum_{k=1}^n \sum_{i=1}^{d-1} \binom{k+d-i-1}{d-i-1} \bigg(
\sum_{j=1}^{d-i} \binom{d-i-1}{j-1} \big(s_j(M) +
s_{j+1}(M)\big)\bigg)\\
&=& \sum_{k=1}^n \sum_{i=1}^{d-1} \binom{k+d-i-1}{d-i-1} \bigg(
\sum_{j=1}^{d-i+1}\binom{d-i}{j-1} s_j(M) \bigg)\\
&=& \sum_{i=1}^{d-1} \sum_{k=1}^n \binom{k+d-i-1}{d-i-1} \bigg(
\sum_{j=1}^{d-i+1}\binom{d-i}{j-1} s_j(M) \bigg) \\
&=& \sum_{i=1}^{d-1} \bigg(\binom{n+d-i}{d-i}- 1\bigg) \bigg(
\sum_{j=1}^{d-i+1}\binom{d-i}{j-1} s_j(M) \bigg) \\
&=& \bigg[\sum_{i=1}^{d-1} \binom{n+d-i}{d-i} \big(
\sum_{j=1}^{d-i+1}\binom{d-i}{j-1} s_j(M) \big) \bigg] -
\sum_{i=1}^{d-1} \sum_{j=1}^{d-i+1}\binom{d-i}{j-1} s_j(M) \\
&=& \bigg[\sum_{i=1}^{d} \binom{n+d-i}{d-i} \big(
\sum_{j=1}^{d-i+1}\binom{d-i}{j-1} s_j(M) \big) \bigg] -
\sum_{i=1}^{d} \sum_{j=1}^{d-i+1}\binom{d-i}{j-1} s_j(M) \\
&\le& \sum_{i=1}^{d} \binom{n+d-1}{d-i} \big(
\sum_{j=1}^{d-i+1}\binom{d-i}{j-1} s_j(M) \big) -
\mathrm{ir}_M(\frak q).
\end{eqnarray*}
Thus
 $$\mathrm{ir}_M(\frak q^{n+1}M) \le \sum_{i=1}^{d} \binom{n+d-i}{d-i} \big( \sum_{j=1}^{d-i+1} \binom{d-i}{j-1} s_j(M)\big).$$
This completes the proof.
\end{proof}
\begin{remark}\label{R3.4} \rm It should be noted that the inequality in Theorem \ref{ThA} becomes an quality
$$\mathrm{ir}_M(\frak q^{n+1}M) = \sum_{i=1}^{d} \binom{n+d-i}{d-i} \big( \sum_{j=1}^{d-i+1} \binom{d-i}{j-1} s_j(M)\big) + s_0(M)$$
for some $n \ge 1$ if
$$\mathrm{ir}_{\overline{M}}(\frak q) = \sum_{i = 1}^d \binom{d}{i}s_i(M),$$
where $\overline{M} = M/H^0_{\frak m}(M)$. In the next section we will show that the converse conclusion  is also true.
\end{remark}

Let $\fq$ be a parameter ideal of $M$.  By Theorem \ref{T2.8}  we write for large enough $n$
$$\mathrm{ir}_M(\fq^{n+1}M)=\sum\limits_{i=0}^{d-1} (-1)^{i}f_i(\fq;M)\binom{n+d-i-1}{d-i-1}$$
with $f_i(\fq;M) \in \mathbb{Z}$ for all $i = 0, \ldots, d-1$. The following immediate consequence of Theorem \ref{ThA} is a generalization of \cite[Proposition 3.4]{Tr14}.
\begin{corollary}\label{cor} Let $M$ be a generalized Cohen-Macaulay module of dimension $d \ge 2$ and $\frak q$ a standard parameter ideal. Then
$$f_0(\frak q;M)\le \sum_{j=1}^{d} \binom{d-1}{j-1} s_j(M).$$
\end{corollary}

 In the next section we can see that the inequality in Corollary  \ref{cor} is an equality for all  parameter ideals contained in a large enough power of $\fm$.
\section{The extremely case}
In this section we compute the index of reducibility of powers of standard parameter ideals $\frak q$ that satisfy the extremely condition, it means that
$$\mathrm{ir}_M(\frak q) = \sum_{i = 0}^d \binom{d}{i}s_i(M).$$
\begin{lemma}\label{L4.1} Let $M$ be a generalized Cohen-Macaulay modules of dimension $d$ and $\frak q = (x_1, ..., x_d)$ a parameter ideal of $M$. Let $M' = M/x_1M$ and $\frak q' = (x_2, ..., x_d)$. Then
\begin{enumerate}[{(i)}]  \rm
\item {\it If $\frak q$ is standard and $\mathrm{ir}_M(\frak q) = \sum_{i = 0}^d \binom{d}{i}s_i(M)$, then $s_i(M') = s_i(M)+ s_{i+1}(M)$ for all $i = 0, ..., d-1$ and $\mathrm{ir}_{M'}(\frak q') = \sum_{i = 0}^{d-1} \binom{d-1}{i}s_i(M')$.}
 \item {\it Let $n_0$ be a positive integer such that $\frak m^{n_0}H^i_{\frak m}(M) = 0$ for all $i = 0, ..., d-1$. Then for all parameter ideal $\frak q \subseteq \frak m^{2n_0}$ we have $\frak q$ is standard and
 $$\mathrm{ir}_M(\frak q) = \sum_{i = 0}^d \binom{d}{i}s_i(M).$$}
\end{enumerate}
\end{lemma}
\begin{proof} (i) is trivial, and (ii) follows from \cite[Corollary 4.3]{CQ11}.
  \end{proof}
\begin{lemma}\label{L4.2}
  Let $M$ be a generalized Cohen-Macaulay module of dimension
$d$ and $\overline{M} = M/H^0_{\frak m}(M)$. Let $\frak q = (x_1, ...,x_d)$ be a standard parameter ideal of $M$ satisfying that $\mathrm{ir}_M(\frak q) = \sum_{i = 0}^d \binom{d}{i}s_i(M)$. Then
$$\mathrm{ir}_M(\frak q^{n+1}M) = \mathrm{ir}_{\overline{M}}(\frak q^{n+1}\overline{M}) + s_0(R)$$
for all $n \ge 0$.
\end{lemma}
\begin{proof} The case $n>0$ was proved in Lemma \ref{L3.1}. The case $n = 0$, we have $\mathrm{ir}_M(\frak q) \le \mathrm{ir}_{\overline{M}} (\frak q) + s_0(M)$ and
$$\mathrm{ir}_{\overline{M}} (\frak q) \le  \sum_{i = 0}^d \binom{d}{i}s_i(\overline{M}) =  \sum_{i = 1}^d \binom{d}{i}s_i(M).$$
Thus $\mathrm{ir}_{\overline{M}} (\frak q) =  \sum_{i = 1}^d \binom{d}{i}s_i(M)$ and $\mathrm{ir}_M(\frak q) = \mathrm{ir}_{\overline{M}} (\frak q) + s_0(M)$.
\end{proof}
Similarly Lemma \ref{L3.3} we have the following result.
\begin{lemma} \label{L4.3}
  Let $M$ be a generalized Cohen-Macaulay module of dimension
$d>1$ and $\frak q = (x_1, ...,x_d)$  a standard parameter ideal of $M$ satisfying that $\mathrm{ir}_M(\frak q) = \sum_{i = 0}^d \binom{d}{i}s_i(M)$. Set $M' = M/x_1M$ and $\frak q' = (x_2, ..., x_d)$. Then
$$\mathrm{ir}_M(\frak q^{n+1}M) - \mathrm{ir}_M(\frak q^{n}M) \le \mathrm{ir}_{M'}((\frak q')^{n+1}M') - (s_0(M) + s_1(M))$$
for all $n \ge 1$.
\end{lemma}
\begin{proof}
  By Lemma \ref{L2.2} we have $\frak q^{n+1}M : x_1 = \frak q^nM + H^0_{\frak m}(M)$, so we have the short exact sequence
$$0 \to \overline{M}/\frak q^n\overline{M} \to M/\frak q^{n+1}M \to M'/(\frak q')^{n+1}M' \to 0$$
  for all $n \ge 1$, where $\overline{M} = M/H^0_{\frak m}(M)$. By applying the functor
$\mathrm{Hom}_R(R/\frak m, \bullet)$ we obtain
$$0 \to (\frak q^n\overline{M}:_{\overline{M}}\frak m)/\frak q^n\overline{M} \to
(\frak q^{n+1}M:_M\frak m)/\frak q^{n+1}M \overset{\varphi_n}{\to}
[(\frak q')^{n+1}M' :_{M'} \frak  m]/(\frak q')^{n+1}M' .$$
Therefore
$$\mathrm{ir}_M(\frak q^{n+1}M) - \mathrm{ir}_{\overline{M}}(\frak q^{n}\overline{M}) = \mathrm{ir}_{M'}((\frak q')^{n+1}M') - \ell (\mathrm{coker}(\varphi_n))$$
for all $n \ge 1$. On the other hand we have $\mathrm{im}(\varphi_n) = [(\frak q^{n+1}M:_M\frak m)+ x_1M]/(\frak q^{n+1}M+x_1M)$, hence
$$\mathrm{coker}(\varphi_n) = \frac{(\frak q^{n+1}M+x_1M):_M\frak m}{(\frak q^{n+1}M:_M \frak m)+ x_1M}.$$
Thus for all $n \ge 1$ we have
\begin{eqnarray*}
\ell(\mathrm{coker}(\varphi_n)) &\ge & \ell \big( \frac{(\frak q^{n+1}M:_M \frak m)+ (x_1M:_M \frak m)}{(\frak q^{n+1}M:_M \frak m)+ x_1M} \big)\\
&=& \ell \big( \frac{ x_1M:_M \frak m}{( x_1M:_M \frak m) \cap [(\frak q^{n+1}M:_M \frak m)+ x_1M]} \big)\\
&=& \ell \big( \frac{ x_1M:_M \frak m}{[( x_1M:_M \frak m) \cap (\frak q^{n+1}M:_M \frak m)]+ x_1M} \big).
\end{eqnarray*}
{\bf Claim.} $( x_1M:_M \frak m) \cap (\frak q^{n+1}M:_M \frak m) \subseteq x_1M + (0:_M \frak m)$.\\
Indeed, let $x \in ( x_1M:_M \frak m) \cap (\frak q^{n+1}M:_M \frak m)$. Since $\frak q^{n+1}M:_M \frak m \subseteq \frak q^{n+1}M:_M x_1 = \frak q^nM + H^0_{\frak m}(M)$ we have $x \in ( x_1M:_M \frak m) \cap (\frak qM + H^0_{\frak m}(M))$. Therefore $x = y + z$ for some $y \in \frak qM$ and $z \in H^0_{\frak m}(M)$. Since $\frak mx$ and $\frak my$ are subsets of $\frak qM$ we have $\frak mz \subseteq H^0_{\frak m}(M) \cap \frak qM = 0$ by Lemma \ref{L2.2} (ii). Hence $z \in 0:_M \frak m$. Therefore
\begin{eqnarray*}
( x_1M:_M \frak m) \cap (\frak q^{n+1}M:_M \frak m) &\subseteq & ( x_1M:_M \frak m) \cap [\frak qM + (0:_M \frak m)] \\
&= & (0:_M \frak m) + [( x_1M:_M \frak m) \cap \frak qM] = x_1M + (0:_M \frak m).
\end{eqnarray*}
The claim is proved.  \\
It is clear that $ 0:_M \frak m \subseteq ( x_1M:_M \frak m) \cap (\frak q^{n+1}M:_M \frak m)$. By the claim we have
$$ ( x_1M:_M \frak m) \cap (\frak q^{n+1}M:_M \frak m) + x_1M = x_1M + (0:_M \frak m).$$
Thus
$$\ell(\mathrm{coker}(\varphi_n)) \ge \ell \big( \frac{ x_1M:_M \frak m}{x_1M + (0:_M \frak m)} \big) = s_0(M') - s_0(M) = s_1(M).$$
By Lemma \ref{L4.2} we have $\mathrm{ir}_{\overline{M}}(\frak q^{n}\overline{M}) = \mathrm{ir}_{M}(\frak q^{n}M) - s_0(M)$. Therefore
$$\mathrm{ir}_M(\frak q^{n+1}M) - (\mathrm{ir}_{M}(\frak q^{n}M) - s_0(M)) \le \mathrm{ir}_{M'}((\frak q')^{n+1}M') - s_1(M)$$
for all $n \ge 1$. The proof is complete.
\end{proof}

\begin{remark} \label{R4.4}\rm According to the proof of Lemma \ref{L4.3} we can see that the inequality of Lemma \ref{L4.3} becomes an equality if and only if
$$(\frak q^{n+1}M+x_1M):_M\frak m = (\frak q^{n+1}M:_M \frak m)+ (x_1M:_M \frak m)$$ for all $n \ge 1$.
  \end{remark}
In order to prove Theorem \ref{ThB} need the following key lemma.
\begin{lemma}\label{L4.5}
  Let $M$ be a generalized Cohen-Macaulay module of dimension
$d>0$ and $\frak q = (x_1, ...,x_d)$ a standard parameter ideal of $M$. Suppose that
 $$\mathrm{ir}_M(\frak q^{n+1}M) = \sum_{i=1}^{d} \binom{n+d-i}{d-i} \big( \sum_{j=1}^{d-i+1} \binom{d-i}{j-1} s_j(M)\big) + s_0(M)$$
for all $n \ge 0$. Then
$$\frak q^{n+1}M :_M \frak m = \frak q^n (\frak qM :_M \frak m) + (0 :_M \frak m)$$
for all $n \ge 0$.
\end{lemma}
\begin{proof} We will proceed by induction on $d$. If $d = 1$ and $\frak q = (x)$. We only need to prove that $x^{n+1}M :_M \frak m \subseteq x^n(xM:_M \frak m) + 0:_M \frak m$ for all $n \ge 0$. The case $n = 0$ is trivial so we can assume that $n \ge 1$. Let $a \in x^{n+1}M :_M \frak m$. Since $x^{n+1}M:_M \frak m \subseteq x^{n+1}M:_M x = x^nM + H^0_{\frak m}(M)$, we have $a = x^nb + c$ for some $b \in M$ and $c \in H^0_{\frak m}(M)$. Since $\frak ma \subseteq x^{n+1}M$ and $\frak m x^nb \subseteq x^nM$ we have $\frak mc \subseteq xM \cap H^0_{\frak m}(M) = 0$. Thus $c \in 0:_M \frak m$. Therefore $x^n \frak m b  = \frak m a \subseteq x^{n+1}M$. Hence $\frak mb \subseteq xM + H^0_{\frak m}(M)$. On the other hand by Lemma \ref{L4.2} we have $\mathrm{ir}_M(x) = \mathrm{ir}_{M/H^0_{\frak m}(M)}(x) + s_0(M)$. So the map
$$\mathrm{Soc}(M/xM) \to \mathrm{Soc}(M/(xM+H^0_{\frak m}(M)) $$
is surjective.Therefore $ b \in (xM + H^0_{\frak m}(M)):_M \frak m = (xM:_M \frak m) + H^0_{\frak m}(M))$. Thus $b = d + e$ with some $d \in xM:_M \frak m$ and $e \in H^0_{\frak m}(M)$. Conclusion, we have
$$a = x^{n}(d+e) + c = x^n d + c \in x^n(xM:_M \frak m) + 0:_M \frak m.$$
Hence $x^{n+1}M :_M \frak m \subseteq x^n(xM:_M \frak m) + 0:_M \frak m$ as desired.\\
Suppose $d>1$ and the assertion is proved for $d-1$. Let $M' = M/x_1M$ and $\frak q' = (x_2, ..., x_d)$. By Lemma \ref{L4.1} we have $\mathrm{ir}_{M'}(\frak q') = \sum_{i=0}^{d-1} \binom{d-1}{i}s_i(M')$. By Lemma \ref{L4.3}, for all $n \ge 1$, we have
$$\mathrm{ir}_M(\frak q^{n+1}M) - \mathrm{ir}_M(\frak q^{n}M) \le \mathrm{ir}_{M'}((\frak q')^{n+1}M') - (s_0(M) + s_1(M)) = \mathrm{ir}_{M'}((\frak q')^{n+1}M') - s_0(M').$$

By our assumption that
$$\mathrm{ir}_M(\frak q^{n+1}M) = \sum_{i=1}^{d} \binom{n+d-i}{d-i} \big( \sum_{j=1}^{d-i+1} \binom{d-i}{j-1} s_j(M)\big) + s_0(M)$$
for all $n \ge 0$ we have
$$\mathrm{ir}_{M'}((\frak q')^{n+1}M') \ge \sum_{i=1}^{d-1} \binom{n+d-1-i}{d-1-i} \big( \sum_{j=1}^{d-i} \binom{d-1-i}{j-1} s_j(M')\big) + s_0(M')$$
for all $n \ge 0$. Combining with Theorem \ref{ThA} we have
$$\mathrm{ir}_{M'}((\frak q')^{n+1}M') = \sum_{i=1}^{d-1} \binom{n+d-1-i}{d-1-i} \big( \sum_{j=1}^{d-i} \binom{d-1-i}{j-1} s_j(M')\big) + s_0(M')$$
for all $n \ge 0$, and
$$\mathrm{ir}_M(\frak q^{n+1}M) - \mathrm{ir}_M(\frak q^{n}M) = \mathrm{ir}_{M'}((\frak q')^{n+1}M') - (s_0(M) + s_1(M)).$$

 Now by the inductive hypothesis we have
$$(x_1M + \frak q^{n+1}M):_M \frak m = \frak q^n (\frak qM :_M \frak m) + (x_1M :_M \frak m)$$
for all $n \ge 0$. On the other hand by Remark \ref{R4.4} we have
$$(x_1M + \frak q^{n+1}M):_M \frak m = (\frak q^{n+1}M :_M \frak m) + (x_1M :_M \frak m)$$
for all $n \ge 0$. Therefore
$$(\frak q^{n+1}M :_M \frak m) + (x_1M :_M \frak m) = \frak q^n (\frak qM :_M \frak m) + (x_1M :_M \frak m)$$
for all $n \ge 0$. Hence
$$(\frak q^{n+1}M :_M \frak m) = \frak q^n (\frak qM :_M \frak m) + [(x_1M :_M \frak m) \cap (\frak q^{n+1}M :_M \frak m)]$$
By the Claim of the proof of Lemma \ref{L4.3} we have
$$(\frak q^{n+1}M :_M \frak m) \subseteq \frak q^n (\frak qM :_M \frak m) + (0:_M \frak m) + x_1M.$$ We are now ready to prove $\frak q^{n+1}M :_M \frak m = \frak q^n (\frak qM :_M \frak m) + (0 :_M \frak m)$ for all $n \ge 0$ by induction on $n$. The case $n = 0$ is trivial. For $n \ge 1$ we have
\begin{eqnarray*}
\frak q^{n+1}M :_M \frak m  &=& \frak q^n (\frak qM :_M \frak m) + (0:_M \frak m) + [(\frak q^{n+1}M :_M \frak m) \cap x_1M]\\
&=& \frak q^n (\frak qM :_M \frak m) + (0:_M \frak m) +  x_1(\frak q^{n+1}M :_M (x_1\frak m))\\
&=& \frak q^n (\frak qM :_M \frak m) + (0:_M \frak m) +  x_1[(\frak q^nM + H^0_{\frak m}(M)) :_M \frak m)] \,\, \text{by Lemma \ref{L2.2}}\\
&=& \frak q^n (\frak qM :_M \frak m) + (0:_M \frak m) +  x_1[(\frak q^nM  :_M \frak m) + H^0_{\frak m}(M)] \,\, \text{by Lemma \ref{L4.2}}\\
&=& \frak q^n (\frak qM :_M \frak m) + (0:_M \frak m) +  x_1(\frak q^nM  :_M \frak m) \,\,\text{by inductive hypothesis}\\
&=& \frak q^n (\frak qM :_M \frak m) + (0:_M \frak m) +  x_1[\frak q^{n-1}(\frak qM  :_M \frak m) + (0:_M \frak m)] \\
 &=& \frak q^n (\frak qM :_M \frak m) + (0:_M \frak m).
\end{eqnarray*}
The proof is complete.
\end{proof}

We prove the main result of this section.
\begin{proof}[Proof of Theorem \ref{ThB}] The case $n =0$ is trivial since
$$\sum_{i=0}^d \binom{d}{i}s_i(M) = \sum_{i=1}^{d} \big( \sum_{j=1}^{d-i+1} \binom{d-i}{j-1} s_j(M)\big) + s_0(M).$$
 Assume that $n>0$, we prove by induction of $d$. If $d=1$ the assertion is clear. Suppose $d>1$ and we have proved the assertion for $d-1$. Let $M' = M/x_1M$ and $\frak q' = (x_2, ..., x_d)$. By Lemma \ref{L4.1} we have $\mathrm{ir}_{M'}(\frak q') = \sum_{i=0}^{d-1} \binom{d-1}{i}s_i(M')$. So by induction we have
$$\mathrm{ir}_{M'}((\frak q')^{n+1}M') = \sum_{i=1}^{d-1} \binom{n+d-1-i}{d-1-i} \big( \sum_{j=1}^{d-i} \binom{d-1-i}{j-1} s_j(M')\big) + s_0(M')$$
for all $n \ge 0$. By Lemma \ref{L4.5} we have
$$(x_1M + \frak q^{n+1}M):_M \frak m = \frak q^n (\frak qM :_M \frak m) + (x_1M :_M \frak m) \subseteq (\frak q^{n+1} M :_M \frak m) + (x_1M :_M \frak m).$$
Therefore $(x_1M + \frak q^{n+1}M):_M \frak m = (\frak q^{n+1} M :_M \frak m) + (x_1M :_M \frak m)$ for all $n \ge 0$. Now Remark \ref{R4.4} implies that
$$\mathrm{ir}_M(\frak q^{n+1}M) - \mathrm{ir}_M(\frak q^{n}M) = \mathrm{ir}_{M'}((\frak q')^{n+1}M') - (s_0(M) + s_1(M))$$
for all $n \ge 0$. Thus for all $n>0$ we have
$$\mathrm{ir}_M(\frak q^{n+1}M) - \mathrm{ir}_M(\frak q) = \sum_{k=1}^n[\mathrm{ir}_{M'}((\frak q')^{k+1}M') - (s_0(M) + s_1(M))].$$
By the same combinatorial transformations used in the proof of Theorem \ref{ThA} we have
 $$\mathrm{ir}_M(\frak q^{n+1}M) = \sum_{i=1}^{d} \binom{n+d-i}{d-i} \big( \sum_{j=1}^{d-i+1} \binom{d-i}{j-1} s_j(M)\big) + s_0(M)$$
for all $n \ge 0$. The proof is complete.
\end{proof}
\begin{corollary}\label{C4.7} Let $M$ be a generalized Cohen-Macaulay
module of dimension $d$. Let $\frak q$ be a standard
parameter ideal such that $\mathrm{ir}_M(\frak q) = \sum_{i=0}^d \binom{d}{i}s_i(M)$.
Then $$\frak q^{n+1}M :_R \frak  m  = \frak q^{n} (\frak qM :_R \frak m)
+ (0 :_M \frak  m)$$ for all $n \geq 0$.
\end{corollary}
\begin{proof} It follows from Lemma \ref{L4.5} and Theorem \ref{ThB}.
\end{proof}
\begin{corollary}\label{C4.8}
Let $M$ be a generalized Cohen-Macaulay
module of dimension $d$. Let $n_0$ be a positive integer such that $\frak m^{n_0}H^i_{\frak m}(M) = 0$ for all $i = 0, ..., d-1$. Then for all parameter ideals $\frak q \subseteq \frak m^{2n_0}$ we have
 $$\mathrm{ir}_M(\frak q^{n+1}M) = \sum_{i=1}^{d} \binom{n+d-i}{d-i} \big( \sum_{j=1}^{d-i+1} \binom{d-i}{j-1} s_j(M)\big) + s_0(M)$$
for all $n \ge 0$.
\end{corollary}
\begin{proof} It follows from Lemma \ref{L4.1} (ii) and Theorem \ref{ThB}.
\end{proof}

\begin{corollary}
  Let $M$ be a generalized Cohen-Macaulay
module of dimension $d$ and $\frak q$ a standard parameter ideal of $M$. Put $\overline{M} = M/H^0_{\frak m}(M)$. Suppose that
$\mathrm{ir}_{\overline{M}}(\frak q) = \sum_{i=1}^d \binom{d}{i}s_i(M)$. Then we have
$$\mathrm{ir}_M(\frak q^{n+1}M) = \sum_{i=1}^{d} \binom{n+d-i}{d-i} \big( \sum_{j=1}^{d-i+1} \binom{d-i}{j-1} s_j(M)\big) + s_0(M)$$
for all $n \ge 1$.
\end{corollary}
\begin{proof} If $\fq$ is a standard parameter ideal of $M$, then it is also a standard parameter ideal of $\overline{M}$. By Lemma \ref{L3.1} we have
$\mathrm{ir}_M(\frak q^{n+1}M) = \mathrm{ir}_{\overline{M}}(\frak q^{n+1}\overline{M}) + s_0(M)$ for all $n \ge 1$. The assertion follows from Theorem \ref{ThB}.
\end{proof}

\begin{corollary}
  Let $M$ be a generalized Cohen-Macaulay
module of dimension $d$ and $\frak q$ a parameter ideal of $M$. Suppose that $\frak q$ is a standard parameter ideal of $\overline{M} = M/H^0_{\frak m}(M)$ and
$\mathrm{ir}_{\overline{M}}(\frak q) = \sum_{i=1}^d \binom{d}{i}s_i(M)$. Then
$$\mathrm{ir}_M(\frak q^{n+1}M) = \sum_{i=1}^{d} \binom{n+d-i}{d-i} \big( \sum_{j=1}^{d-i+1} \binom{d-i}{j-1} s_j(M)\big) + s_0(M)$$
for all $n \gg 0$.
\end{corollary}
\begin{proof}
 By \cite[Lemma 2.4]{CT08} we have $(\frak q^nM + H^0_{\frak m}(M)):_M \frak m = (\frak q^nM:_M \frak m) + H^0_{\frak m}(M)$ for all $n \gg 0$. Then similar to Lemma \ref{L3.1} we have $\mathrm{ir}_M(\frak q^{n+1}M) = \mathrm{ir}_{\overline{M}}(\frak q^{n+1}\overline{M}) + s_0(M)$ for all $n \gg 0$. Now the assertion follows from Theorem \ref{ThB}.
\end{proof}

\begin{example}\rm Let $S = K[[X, Y , Z, W]]$ be the formal power series ring over a field $K$. We look at the following typical Buchsbaum ring
$$R= K[[X, Y , Z, W]]/(X, Y)\cap( Z, W).$$
Let $x$, $y$, $z$, and $w$ denote image of $X$, $Y$, $Z$, and $W$ in $R$, respectively. It is easy to see that $R$ is a Buchsbaum ring with $\dim R=2$ and $\depth R=1$. More precisely, by the exact sequence
	$$0\to R = S/(X, Y)\cap( Z, W)\to S/(X,Y)\oplus S/(Z,W) \to S/(X,Y,Z,W) = K\to 0 $$
one can check that $s_0(R) = 0, s_1(R)=1, s_2(R)=2$. Since $R$ is Buchsbaum, every parameter ideal $\fq$ of $R$ is standard. By Theorems \ref{ThA}, \ref{ThB} and Remark \ref{R3.4} we have
\begin{eqnarray*}
\mathrm{ir}_R(\fq^{n+1}) &=& \sum_{i=1}^{2} \binom{n+2-i}{2-i} \big( \sum_{j=1}^{2-i+1} \binom{2-i}{j-1} s_j(R)\big) + s_0(R)\\
&=& ( s_1(R) + s_2(R))(n+1) + (s_0(R) + s_1(R)).\\
&=& 3(n+1) + 1.
\end{eqnarray*}
if and only if
 $$\mathrm{ir}_R(\fq)=\sum_{i=0}^2 \binom{2}{i}s_i(R) = 4.$$
 By Lemma \ref{L4.1} (ii) this is the case if $\fq \subseteq \fm^2$, where $\fm = (x,y,z,w)$. We product an example with $\mathrm{ir}_R(\fq^{n+1}) < 3(n+1)+1$ for all $n \ge 0$. Let $a=x-z$, $b=y-w$, and $\fq_0=(a,b)$. It is easy to see that
$$R/\fq_0 \cong k[[X,Y]]/(X,Y)^2.$$
Therefore $\fq_0:\frak m=\frak m$ and $\mathrm{ir}_R(\fq_0) = 2 <4$. Thus
$$\mathrm{ir}_R(\fq_0^{n+1}) < 3(n+1)+1$$
for all $n \ge 0$.
\end{example}

 \section{Hilbert polynomials of socle ideals.}
 In this section we assume that $(R, \frak m)$ is a generalized Cohen-Macaulay local ring of dimension $d$. Let $I$ be an $\frak m$-primary ideal. It is well known that
$$\ell (R/I^{n+1})  = e_0(I) \binom{n+d}{d} - e_1(I) \binom{n+d-1}{d-1} + \cdots + (-1)^d e_d(I)$$
for all $n \gg 0$. These integers $e_i(I)$ are called the Hilbert coefficients of $I$. We will compute explicitly all $e_i(I)$ in the case $I = \frak q:_R \frak m$, where $\frak q$ is a standard parameter ideal of $R$ satisfying $\mathrm{ir}_R(\frak q) = \sum_{i=0}^d \binom{d}{i}s_i(R)$.
\begin{lemma}\label{L4.9} Let $(R, \frak m)$ is a generalized Cohen-Macaulay local ring of dimension $d$ such that $R$ is not regular. Assume that $\frak q$ is a standard parameter ideal of $R$ satisfying $\mathrm{ir}_R(\frak q) = \sum_{i=0}^d \binom{d}{i}s_i(R)$. Put $I = \frak q:_R \frak m$. Then for all $n \ge 1$ we have
  $$\ell(I^{n+1}/\frak q^{n+1}) = \sum_{i=1}^{d} \binom{n+d-i}{d-i} \big( \sum_{j=1}^{d-i+1} \binom{d-i}{j-1} s_j(R)\big).$$
\end{lemma}
\begin{proof}
 By \cite[Theorem 1.2]{CT08} we have $I^2 = \frak qI$. Thus $I^{n+1} = \frak q^n (\frak q:_R \frak m)$. By Corollary \ref{C4.7} we have $\ell ((\frak q^{n+1}:_R \frak m)/I^{n+1}) = s_0$ for all $n \ge 1$. So
 $$\ell(I^{n+1}/\frak q^{n+1}) = \mathrm{ir}_R(\frak q^{n+1}) - s_0(R)$$
 for all $n \ge 1$. The assertion follows from Theorem \ref{ThB}.
\end{proof}
\begin{theorem}\label{T4.10}
  Let $(R, \frak m)$ is a generalized Cohen-Macaulay local ring of dimension $d>0$ such that $R$ is not regular. Let $\frak q$ be a standard parameter ideal of $R$ satisfying $\mathrm{ir}_R(\frak q) = \sum_{i=0}^d \binom{d}{i}s_i(R)$. Put $I = \frak q:_R \frak m$, and $h_j(R) = \ell(H^j_{\frak m}(R))$ for all $j = 0, ..., d-1$. Then we have $e_0(I) = e_0(\fq)$ and
  $$e_i(I) = (-1)^i \big(\sum_{j =1 }^{d-i} \binom{d-i-1}{j-1}h_j(R) - \sum_{j=1}^{d-i+1} \binom{d-i}{j-1} s_j(R) \big)$$
  for all $i = 1, ...,d-1$, and $e_d(I) = (-1)^d(h_0 - s_1)$.
\end{theorem}
\begin{proof} We have
 $$\ell(R/I^{n+1}) = \ell(R/\frak q^{n+1}) - \ell (I^{n+1}/\frak q^{n+1}).$$
 By \cite[Korollar 3.2]{Sc79} (see also \cite[Theorem 4.2]{T86}) we have
 $$\ell(R/\frak q^{n+1}) = \binom{n+d}{d}e_0(\frak q) + \sum_{i=1}^d  \binom{n+d-i}{d-i} \sum_{j =0 }^{d-i} \binom{d-i-1}{j-1}h_j(R)$$
 for all $n \ge 0$, where $\binom{d-i-1}{-1} = 0$ if $i \neq d$, and $\binom{-1}{-1} = 1$. Combining with Lemma \ref{L4.9} we have
 $$\ell(R/I^{n+1}) = \binom{n+d}{d}e_0(\frak q) + \sum_{i=1}^d  \binom{n+d-i}{d-i} \bigg[ \sum_{j =0 }^{d-i} \binom{d-i-1}{j-1}h_j(R) - \sum_{j=1}^{d-i+1} \binom{d-i}{j-1} s_j(R)\bigg]$$
 for all $n \ge 1$.\\
 Now, the assertion is just comparing coefficients of two polynomials.
\end{proof}
Notice that $s_i(R) = h_i(R)$ for all $i = 0, ..., d-1$, provided $R$ is Buchsbaum. By some combinatorial transformations we have the following. The readers are encouraged to compare this result with \cite[Theorem 1.1]{GJS10}.
\begin{corollary}
  Let $(R, \frak m)$ is a Buchsbaum local ring of dimension $d>1$ such that $R$ is not regular. Suppose that $\frak q$ be a parameter ideal of $R$ satisfying that $\mathrm{ir}_R(\frak q) = \sum_{i=0}^d \binom{d}{i}s_i(R)$. Set $I = \frak q:_R \frak m$. Then we have $e_0(I) = e_0(\fq)$,
  $$e_1(I) = \sum_{j =1 }^{d-1} \binom{d-2}{j-2}h_j(R) + s_d(R)$$
  and
  $$e_i(I) = (-1)^{i+1} \sum_{j=1}^{d-i+1} \binom{d-i-1}{j-2}h_j(R)$$
  for all $i = 2, ...,d-1$, and
  $$e_d(I) = (-1)^d(h_0(R) - h_1(R)).$$
\end{corollary}
\begin{corollary}
Let $(R, \frak m)$ is a Cohen-Macaulay local ring of dimension $d>0$ such that $R$ is not regular. Let $\frak q$ be a parameter ideal of $R$, and $I = \frak q:_R \frak m$ the socle ideal. Then we have $e_0(I) = e_0(\fq)$ and $e_1(I) = s_d$.
\end{corollary}

\end{document}